\documentclass[pdflatex,sn-mathphys-num]{sn-jnl}

\usepackage{tikz}
\usepackage{pgfplots}
\pgfplotsset{compat=newest}
\usepackage{float}
\usepackage{url}
\usepackage{placeins}

\usepackage{graphicx}%
\usepackage{multirow}%
\usepackage{amsmath,amssymb,amsfonts}%
\usepackage{amsthm}%
\usepackage{mathrsfs}%
\usepackage[title]{appendix}%
\usepackage{xcolor}%
\usepackage{textcomp}%
\usepackage{manyfoot}%
\usepackage{booktabs}%
\usepackage{algorithm}%
\usepackage{algorithmicx}%
\usepackage{algpseudocode}%
\usepackage{listings}%

\usepackage{geometry}
\usepackage{fancyhdr}
\fancypagestyle{arxivstyle}{
    \fancyhf{}

    \fancyhead[L]{\small Positive-Allocation Companion Predictions}
    \fancyhead[R]{\small Flores and Pascoe}

    \fancyfoot[C]{\thepage}

}

\theoremstyle{thmstyleone}%
\newtheorem{theorem}{Theorem}
\newtheorem{proposition}[theorem]{Proposition}%

\theoremstyle{thmstyletwo}%
\newtheorem{remark}{Remark}%

\theoremstyle{thmstylethree}%

\begin{document}

\title[Positive-Allocation Companion Predictions]{Positive-Allocation Companion Predictors for Nonlinear Dynamics and Their Finite-Difference Diagnostics}

\author*[1]{\fnm{Cynthia} \sur{Flores}}\email{cynthia.flores@csuci.edu}

\author[2]{\fnm{James} \sur{Pascoe}}\email{jep362@drexel.edu}

\affil*[1]{\orgdiv{Department of Mathematics, Data Science, and Statistics}, \orgname{California State University Channel Islands}, \orgaddress{\street{One University Drive}, \city{Camarillo}, \postcode{93012}, \state{CA}, \country{USA}}}

\affil[2]{\orgdiv{Department of Mathematics}, \orgname{Drexel University}, \orgaddress{\street{3141 Chestnut Street}, \city{Philadelphia}, \postcode{19104}, \state{PA}, \country{USA}}}

\abstract{We introduce a positive-allocation companion construction for Koopman-inspired finite-dimensional prediction of nonlinear dynamical systems. The method determines recurrence coefficients by representing a target observable snapshot as a nonnegative, normalized combination of earlier training snapshots. These coefficients define a companion matrix whose spectral structure is induced by the allocation constraints at the construction stage. We prove that normalized positive allocation places the companion spectrum in the closed unit disk and, because the coefficients sum to one, includes $1$ as an eigenvalue. Additionally, we develop modal and non-modal diagnostics for the resulting model trajectory. When the companion matrix is diagonalizable, the modal representation shows that first and second finite differences act as spectral filters through factors of $\lambda_\ell-1$ and $(\lambda_\ell-1)^2$. We also derive $C$-based finite-difference bounds that avoid diagonalization and can be evaluated directly from the training data and companion matrix. Numerical experiments on the FitzHugh--Nagumo and susceptible--infectious--recovered (SIR) models illustrate the behavior of the construction in oscillatory and transient dissipative settings. The examples demonstrate both the interpretability of the companion recurrence and its limitations, particularly when pointwise trajectory agreement degrades while finite-difference and modal diagnostics remain informative.
}



\maketitle
\pagestyle{arxivstyle}

\section{Introduction}\label{sec:Intro}

Forecasting nonlinear dynamical systems is central to applied science and engineering. When the governing equations are unknown or too complex for direct use, data-driven methods provide alternative approaches to prediction and model discovery. These include sparse identification of nonlinear dynamics (SINDy) \cite{BrPrKu2016, ChLuKuBr2019}, neural ordinary differential equations, reservoir computing, and learned Koopman embeddings \cite{LuKuBr2018, ChRuBeDu2018, PaHuGiLuOt2018}. Such methods offer different balances among expressivity, interpretability, and computational tractability, but their performance may depend strongly on the quality and quantity of available data, model architecture or observable selection, numerical conditioning, and forecasting horizon \cite{Ljung1999, PaHuGiLuOt2018, Ru2019, HeRoClDeCa2017, MaWe2024, NuPePhScWo2023}.

Recently, there has been intense interest in Koopman operator theory as a way to represent nonlinear dynamics through a linear operator acting on observables. It is typical that the space of observables is infinite-dimensional, and therefore the Koopman operator $\mathcal{K}$ is itself also generally infinite-dimensional, even when the underlying dynamical system is finite-dimensional. Moreover, $\mathcal K$ is a linear operator that acts on observables, which are possibly nonlinear functions of the system state, rather than on the state itself. The operator $\mathcal{K}$ evolves the observables forward in time under the dynamics, allowing nonlinear systems to be studied using the tools of linear operator theory. Thus, the Koopman framework offers a global linearization of nonlinear dynamics, with the tradeoff that this is happening in an infinite-dimensional function space. 

In practical applications, finite-dimensional approximations of $\mathcal{K}$ are necessary. Spectral representations have proved useful for prediction, control, and system analysis in areas including fluid dynamics, power systems, and robotics \cite{BrBuKaKu2022, RoMeBaScHe2009, Me2013, CoDrHo2025, NEURIPS2022_KosticEtAl}; an overview of learning and control using Koopman models is provided in \cite{VeSoHi2021}. Common computational approaches include Extended Dynamic Mode Decomposition (EDMD) \cite{TuRoLuBrKu2014, WiKero2015}, delay-coordinate and Hankel-DMD \cite{ HaMe2017}, and kernel-based approximations \cite{WiRoKe2015}. Koopman methods have also been applied beyond classical prediction and control settings, including anomaly detection through occupation-kernel methods for quadcopter data~\cite{MoRuLiKa2023} and learned architectures such as KoopAGRU for time-series anomaly detection  \cite{AiBe2025}. These developments build on the broader perspective of applied Koopmanism \cite{BuMoMe2012}, which emphasizes spectral representations as a unifying framework for the analysis of nonlinear dynamics from data.

Recent advances in spectral and modal decomposition have further refined our understanding of Koopman operator approximations and their use in dynamical system prediction. Notably, Korda and Mezi\'{c} introduced operator-theoretic predictors grounded in convex optimization \cite{MiMe2020}, while Arbabi and Mezi\'{c} showed how dynamic mode decomposition and ergodic-theoretic ideas can be used to compute spectral properties of the Koopman operator, including through delay coordinates \cite{HaMe2017}. The work presented by Klus et al. in \cite{KlNuPeNiClSc2020} provides a unified framework for data-driven computation of Koopman generators and model reduction, while the volume \cite{MaMeSu2020} provides an overview of key concepts, computational techniques, and control applications grounded in spectral theory. These developments demonstrate the central role of spectral structure in the construction and interpretation of finite-dimensional models of nonlinear dynamics \cite{HaMe2017, MiMe2020, KlNuPeNiClSc2020, MaMeSu2020}. Spectral representations have also supported interpretable prediction and control in application areas such as health analytics and robotics \cite{NEURIPS2022_KosticEtAl, MaCaTaMu2021}. These works illustrate how modes, eigenfunctions, or learned coordinates can provide structured representations of nonlinear dynamics.

At the same time, rigorous Koopman spectral approximation requires care. Finite-dimensional spectra obtained from DMD- or EDMD-type matrices should not be identified uncritically with the spectrum of the underlying infinite-dimensional Koopman operator, since such approximations may exhibit spectral pollution or fail to represent the true operator spectrum. Residual dynamic mode decomposition and related approaches address these issues through residual-based diagnostics and verified computation of Koopman spectra, pseudospectra, and spectral measures, with convergence guarantees and control of spectral pollution \cite{colbrook2024rigorous, colbrook2023residual}.

The present work is complementary to this line of analysis. Rather than claiming a convergent approximation of the Koopman spectrum, we study a positive-allocation companion predictor whose finite-dimensional spectrum is induced by coefficient constraints at the construction stage. We then analyze the internal modal and finite-difference structure of the resulting model trajectory.

Here, positive allocation refers to the construction of recurrence coefficients subject to nonnegativity and normalization constraints. Rather than first fitting a finite-dimensional linear predictor and subsequently examining its spectrum, the proposed approach determines recurrence coefficients from a constrained least-squares problem and encodes them in a companion matrix. The coefficient constraints provide qualitative spectral information at the construction stage, while kernel-regularized least squares provides a flexible way to compute the allocation in an implicit feature space. The novelty of the approach lies in using coefficient constraints to shape the finite-dimensional spectral structure of the companion predictor and then analyzing the internal modal and finite-difference behavior of the resulting model trajectory. A complete discussion of the construction is given in Section~\ref{sect:pa}.

Having constructed the predictor, we examine analytical consequences of its finite-dimensional spectral representation. In particular, we consider two families of bounds:

\begin{enumerate}
\item modal bounds, obtained from a diagonalization of the companion matrix $C$, when such a diagonalization is available, and expressed in terms of the resulting eigenvalues and modes; and

\item $C$-based (non-modal) bounds, which avoid diagonalization and instead bound the difference between successive powers of $C$ directly.
\end{enumerate}

These two perspectives describe properties of the constructed model trajectory and clarify how the induced companion spectrum influences finite-difference behavior. We also discuss numerical conditioning and illustrate the resulting predictor on nonlinear dynamical systems. The numerical examples are intended to demonstrate the construction and to examine the behavior induced by the positive-allocation recurrence.

Throughout the paper, $\|\cdot\|_2$ denotes the Euclidean norm on finite-dimensional vector spaces. When applied to a matrix, $\|\cdot\|_2$ denotes the induced matrix norm,
$$
\|A\|_2=\sup_{\|v\|_2=1}\|Av\|_2.
$$ The paper is organized as follows. Section~\ref{sec:BP} introduces the Koopman setting, the positive-allocation construction, and the modal representation of the companion predictor. Section~\ref{sec:bounds} develops modal and non-modal finite-difference bounds, with $C$-based bounds in Section~\ref{sect:nonmodal}. Numerical results are presented in Section~\ref{sect:numerical}, and conclusions follow in Section~\ref{sec:conclusions}.

\section{Background and Preliminaries}
\label{sec:BP}

Let $\Omega\subseteq\mathbb{R}^d$ denote the state space, and let $F:\Omega\to\Omega$ be a discrete-time state map. Given an initial state $x_0\in\Omega$, the dynamics generate the forward orbit
$$
x_{j+1}=F(x_j),\qquad j\geq 0.
$$
Thus, $x_j$ denotes the state of the system at the $j$-th time step. The map $F$ may be nonlinear and need not be known explicitly in the data-driven setting considered here.

The Koopman operator framework provides a linear perspective on these possibly nonlinear dynamics. Rather than evolving the state directly, the Koopman operator acts on scalar-valued observables $g:\Omega\to\mathbb{C}$. We denote the Koopman operator by $\mathcal{K}$, defined by composition as
\begin{equation}
    (\mathcal{K}g)(x)=g(F(x)).
    \label{eq:calK}
\end{equation}
Although $F$ may be nonlinear, $\mathcal{K}$ is linear. Because the space of observables is generally infinite-dimensional, $\mathcal{K}$ is itself typically infinite-dimensional, even when $\Omega\subseteq\mathbb{R}^d$ is finite-dimensional. Practical data-driven methods therefore work with finite collections of observables and corresponding finite-dimensional representations; see \cite{Me2005, Me2013, TuRoLuBrKu2014, BuMoMe2012}.

To distinguish a finite collection of observables from the scalar observable $g$ appearing in \eqref{eq:calK}, let
$$
\mathbf{g}(x)
=
\begin{pmatrix}
g_1(x)\\
\vdots\\
g_m(x)
\end{pmatrix}
\in\mathbb{R}^m
$$ denote a selected vector of observables. For each state $x_j$, define the corresponding observable snapshot by
$$
y_j:=\mathbf{g}(x_j).
$$
The observable data matrix is then
$$
Y=[y_0,y_1,\ldots,y_n]\in\mathbb{R}^{m\times(n+1)}.
$$
Thus, each column of $Y$ is associated with an underlying state $x_j$. When the coordinate functions of the state are included among the selected observables, the physical state can be recovered from the corresponding components of $y_j$. In particular, when $m=d$ and $\mathbf{g}$ is the identity map, the observable snapshots coincide with the original state snapshots.

A common objective in data-driven Koopman methods is to identify a finite-dimensional matrix $\widehat K$ such that
$$
y_{j+1}\approx\widehat K y_j.
$$ Extended Dynamic Mode Decomposition (EDMD) constructs such a representation using a finite dictionary of observables \cite{WiKero2015}. Kernel-based variants, including kernel DMD, use a positive-definite kernel to evaluate inner products in an implicit feature space without explicitly constructing the associated feature map \cite{WiRoKe2015}. These approaches provide finite-dimensional representations of the action of the Koopman operator on a selected observable space.

In contrast to these operator-fitting approaches, the construction developed here first identifies a linear recurrence among the observed snapshots and then encodes that recurrence in a companion matrix. The recurrence coefficients are determined through positive allocation. In the kernelized version of the method, the kernel enters the coefficient selection problem rather than defining an EDMD-type operator approximation.

Let
$$
Y_{\mathrm{train}}
=
[y_0,y_1,\ldots,y_{n-1}]
\in\mathbb{R}^{m\times n},
$$ and let
$$
\alpha
=
[\alpha_0,\alpha_1,\ldots,\alpha_{n-1}]^T
\in\mathbb{R}^n.
$$ The coefficient vector $\alpha$ is chosen so that the final observed snapshot $y_n$ is represented, up to a residual, by
\begin{equation}
y_n
=
Y_{\mathrm{train}}\alpha+r(\alpha)
=
\sum_{j=0}^{n-1}\alpha_jy_j+r(\alpha).
\label{eq:allocation_residual}
\end{equation}

The coefficients determine the polynomial
\begin{equation}
p(\lambda)
=
\lambda^n
-
\alpha_{n-1}\lambda^{n-1}
-\cdots
-\alpha_1\lambda
-\alpha_0
\label{eq:poly}
\end{equation}
and the associated companion matrix
\begin{equation}
C=
\begin{pmatrix}
0 & 0 & \cdots & 0 & \alpha_0 \\
1 & 0 & \cdots & 0 & \alpha_1 \\
0 & 1 & \cdots & 0 & \alpha_2 \\
\vdots & \vdots & \ddots & \vdots & \vdots \\
0 & 0 & \cdots & 1 & \alpha_{n-1}
\end{pmatrix}.
\label{eq:matrix_rep}
\end{equation}

Let $e_1=(1,0,\ldots,0)^T\in\mathbb{R}^n$. The corresponding model trajectory is defined by
\begin{equation}
\widehat y_k
=
Y_{\mathrm{train}}C^ke_1,
\qquad k\geq 0.
\label{eq:Cpred}
\end{equation} Here, $\widehat y_k$ denotes the snapshot generated by the companion recurrence at prediction index $k$. The companion matrix $C$ acts on the coefficient coordinates in $\mathbb R^n$, while multiplication by $Y_{\mathrm{train}}$ maps these coordinates back to observable space. Thus, the construction should be distinguished from EDMD- or DMD-type operator-fitting approaches, in which a finite-dimensional matrix is used to advance observable snapshots directly.

Due to the shift structure of $C$,
$$
\widehat y_k=y_k,
\qquad
k=0,\ldots,n-1.
$$ For $k\geq n$, the model trajectory is generated by the recurrence encoded in the final column of $C$. The construction is data-driven and does not require explicit knowledge of the underlying state map $F$.

Section \ref{sect:pa} below specifies how the coefficient vector $\alpha$ is determined and describes the spectral consequences of
the nonnegativity and normalization constraints.

\subsection{Positive Allocation}\label{sect:pa} 

Recall that 
$$
Y_{\mathrm{train}}
=
[y_0,y_1,\ldots,y_{n-1}]
\in \mathbb{R}^{m\times n}
$$ denotes the matrix of observed training snapshots, and let $y_n\in\mathbb{R}^m$ denote the final observed snapshot. The \emph{positive-allocation problem} is the constrained least-squares problem
$$
\min_{\alpha\in\mathbb{R}^n}
\left\|
y_n-Y_{\mathrm{train}}\alpha
\right\|_2^2
$$ subject to
$$
\alpha_j\geq 0,
\qquad
j=0,\ldots,n-1,
$$ and
$$
\sum_{j=0}^{n-1}\alpha_j=1.
$$ We refer to any minimizer
$$
\alpha^*
=
[\alpha_0^*,\alpha_1^*,\ldots,\alpha_{n-1}^*]^T
$$ as a positive allocation of the final snapshot with respect to the training snapshots. Here, ``positive'' is used in the sense of nonnegative coefficients; some entries of $\alpha^*$ may be zero. The normalization constraint implies that
$$
Y_{\mathrm{train}}\alpha^*
=
\sum_{j=0}^{n-1}\alpha_j^*y_j
$$ is a convex combination of the observed training snapshots.

The allocation coefficients determine the recurrence encoded by the companion matrix $C$ in \eqref{eq:matrix_rep}, as well as the characteristic polynomial $p(\lambda)$ in \eqref{eq:poly}. Thus, the spectrum of the companion predictor is induced by the coefficient constraints at the construction stage. Proposition \ref{prop:positiveallocation} describes the resulting spectral containment, while Algorithm \ref{alg:PA} summarizes the numerical construction.

Although sparsity is not imposed explicitly in the optimization problem, positive-allocation solutions may contain many zero or near-zero coefficients. Related sparse behavior has been observed in positive allocation problems arising in modern portfolio theory \cite{best1992positively,hutinet2024indices}. In the present setting, sparsity provides an intuitive interpretation of the recurrence: only a selected collection of earlier snapshots contributes directly to the first predicted snapshot.

Before examining the general spectral consequences of positive allocation, we illustrate this recursive interpretation through a hypothetical sparse allocation. Suppose that the training window consists of the observable
snapshots
$$
y_0,y_1,\ldots,y_{255},
$$ and that the final observed snapshot is $y_{256}$. Thus, the companion matrix has dimension $n=256$.

Consider the hypothetical sparse allocation
$$
\alpha_{128}=\frac{1}{2},
\qquad
\alpha_{192}=\frac{1}{3},
\qquad
\alpha_{5}=\frac{1}{6},
$$ with $\alpha_j=0$ for all remaining indices. These coefficients are nonnegative and satisfy
$$
\frac{1}{2}+\frac{1}{3}+\frac{1}{6}=1.
$$ The first snapshot generated by the recurrence is therefore
\begin{equation}
\widehat y_{256}
=
\frac{1}{2}y_{128}
+
\frac{1}{3}y_{192}
+
\frac{1}{6}y_{5}.
\label{eq:hypothetical}
\end{equation} When the allocation residual is small, $\widehat y_{256}$ provides an approximation of the observed snapshot $y_{256}$.

More generally, for $\gamma\in\mathbb{Z}_{+}$ the recurrence encoded by the companion matrix gives
$$
\widehat y_{256+\gamma}
=
\frac{1}{2}\widehat y_{128+\gamma}
+
\frac{1}{3}\widehat y_{192+\gamma}
+
\frac{1}{6}\widehat y_{5+\gamma}.
$$
For example, taking $\gamma=64$ yields
\begin{align}
\widehat y_{320}
&=
\frac{1}{2}\widehat y_{192}
+
\frac{1}{3}\widehat y_{256}
+
\frac{1}{6}\widehat y_{69}
\nonumber\\
&=
\frac{1}{2}y_{192}
+
\frac{1}{3}
\left(
\frac{1}{2}y_{128}
+
\frac{1}{3}y_{192}
+
\frac{1}{6}y_{5}
\right)
+
\frac{1}{6}y_{69}.
\label{eq:recursion}
\end{align}
Here, $\widehat y_{192}=y_{192}$ and $\widehat y_{69}=y_{69}$ because these indices remain within the training window.

\begin{figure}[htbp]
\centering
\begin{tikzpicture}[scale=0.5, thick]

  \draw[->] (0,0) -- (18,0) node[right] {\small Prediction index};
  \draw[->] (0,0) -- (0,12) node[above] {\small Contributing snapshot index};

  \draw (-0.2,2) -- (0.2,2) node[left=5pt] {\small 5};
  \draw (-0.2,6) -- (0.2,6) node[left=5pt] {\small 128};
  \draw (-0.2,9) -- (0.2,9) node[left=5pt] {\small 192};
  \draw (-0.2,11.5) -- (0.2,11.5) node[left=5pt] {\small 256 max};

  \draw (8,-0.2) -- (8,0.2) node[below=5pt] {\small 256};
  \draw (13,-0.2) -- (13,0.2) node[below=5pt] {\small $256 + \gamma$};

  \draw[dashed] (13,0) -- (13,12);

  \draw[black] (0,0) -- (8,11.5);
  \filldraw[black] (0,0) circle (2.5pt);
  \draw[black, thick, fill=white] (8,11.5) circle (2.5pt);

\foreach \y [count=\i] in {1, 5.5, 9} {
  \pgfmathsetmacro{\xend}{8 + (8/11.5)*(11.5 - \y)}

  \ifnum\i=1 \def\dotcolor{cyan}\fi
  \ifnum\i=2 \def\dotcolor{magenta}\fi
  \ifnum\i=3 \def\dotcolor{blue}\fi

  \draw[black] (8,\y) -- (\xend,11.5);
  \filldraw[black] (8,\y) circle (2.5pt);               
}

\foreach \y [count=\i] in {1, 5.5, 9} {
  \pgfmathsetmacro{\xorig}{(8/11.5)*(11.5 + 11.5 - 9}

  \pgfmathsetmacro{\xend}{\xorig + (8/11.5)*(11.5 - \y)}
  \pgfmathsetmacro{\ystart}{\y}
  \pgfmathsetmacro{\yend}{11.5}

  \draw[blue, thick] (\xorig,\ystart) -- (\xend,\yend);
  \filldraw[blue] (\xorig,\ystart) circle (2.5pt);
   \draw[blue, thin, dashed] (\xorig,0) -- (\xorig,12);
}

\foreach \y [count=\i] in {1, 5.5, 9} {

  \pgfmathsetmacro{\xorigi}{(9.42/11.5)*(11.5 + 11.5 - 9}

  \pgfmathsetmacro{\xend}{\xorigi + (8/11.5)*(11.5 - \y)}
  \pgfmathsetmacro{\ystart}{\y}
  \pgfmathsetmacro{\yend}{11.5}

  \draw[magenta, thick] (\xorigi,\ystart) -- (\xend,\yend);
  \filldraw[magenta] (\xorigi,\ystart) circle (2.5pt);
   \draw[magenta, thin, dashed] (\xorigi,0) -- (\xorigi,12);
}

\foreach \y [count=\i] in {1, 5.5, 9} {
  \pgfmathsetmacro{\xorigi}{(8.01/11.5)*(11.5 + 11.5 - 5.5}

  \pgfmathsetmacro{\xend}{\xorigi + (8/11.5)*(11.5 - \y)}
  \pgfmathsetmacro{\ystart}{\y}
  \pgfmathsetmacro{\yend}{11.5}

  \draw[cyan, thick] (\xorigi,\ystart) -- (\xend,\yend);
  \filldraw[cyan] (\xorigi,\ystart) circle (2.5pt);
   \draw[cyan, thin, dashed] (\xorigi,0) -- (\xorigi,12);
}

\end{tikzpicture}
\caption{Propagation, or ``rebirth,'' of selected snapshot contributions at prediction index $256+\gamma$.}
\label{fig:reborn}
\end{figure}
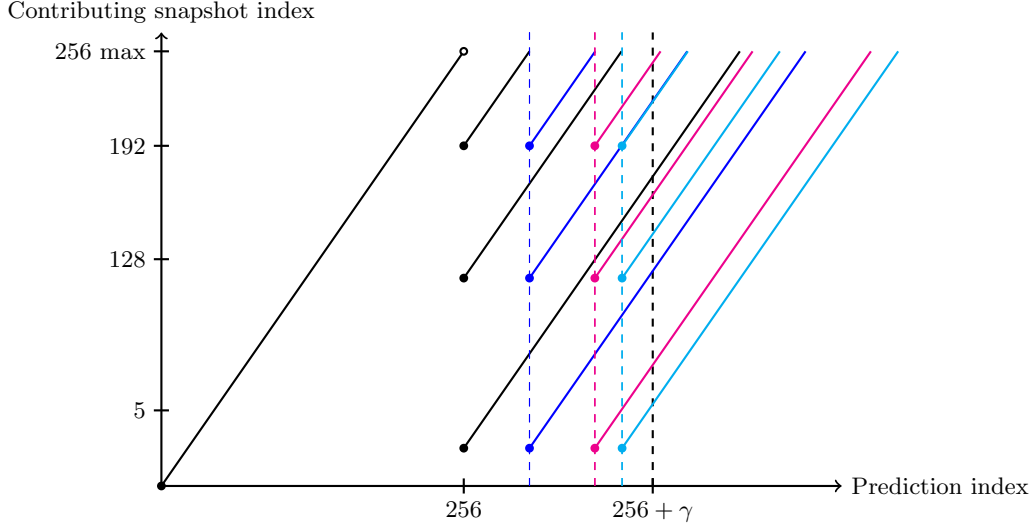

This example shows how the influence of a sparse collection of earlier snapshots propagates through repeated applications of the recurrence. The earlier snapshots are not reproduced literally; rather, their weighted contributions reappear within later model states. In this sense, information from selected portions of the observed trajectory is ``reborn'' as the prediction evolves. Figure~\ref{fig:reborn} illustrates this propagation of snapshot contributions. We now consider the corresponding spectral consequence of the nonnegativity and normalization constraints. In particular, these constraints ensure that the roots of the associated characteristic polynomial lie in the closed unit disk.

\begin{proposition}[Unit-disk containment for positive allocation]
\label{prop:positiveallocation}
Given a polynomial 
\begin{equation}
p(z)=z^n -a_{n-1}z^{n-1} -a_{n-2}z^{n-2}-\hdots - a_1 z - a_0
\label{eq:charpoly}
\end{equation} such that, 
\begin{itemize}
\item All coefficients $a_j\geq 0$ for $j=0,\hdots,n-1$, 
\item The sum of coefficients $\sum_{j=0}^{n-1}a_j =1,$
\end{itemize} then all roots of $p(z)$ lie in the closed unit disk $|z|\leq 1$.
\end{proposition}

\begin{proof}
The result follows from a contradiction argument. Suppose there is a root of $p(z)$ such that $|z|>1.$ Then 
\begin{equation}
z^n=a_{n-1}z^{n-1} +\hdots +a_0 \Rightarrow 1=a_{n-1}z^{-1} +a_{n-2}z^{-2}+\hdots + a_0z^{-n}.
\label{eq:contrapoly}
\end{equation}
Thus, taking magnitudes, we obtain the following inequality,
\begin{equation}
|z^n|=\left| \sum_{j=0}^{n-1} a_j z^j \right| \leq \sum_{j=0}^{n-1} a_j|z|^j =: f(|z|). 
\label{eq:contra}
\end{equation} However, the assumption $|z|>1$ implies $|z|^j<|z|^n$ for every $j<n$. Moreover, $a_j\geq 0$ and $\sum a_j =1$ imply that $$f(|z|)=\sum_{j=0}^{n-1} a_j|z|^j<|z|^n$$ which contradicts the inequality \eqref{eq:contra}. Thus, there cannot be roots outside the unit circle and the desired result is obtained. 
\end{proof}

\begin{remark}
Proposition \ref{prop:positiveallocation} guarantees that normalized nonnegative coefficients produce roots in the closed unit disk, but their precise distribution depends on the coefficients. Certain coefficient families satisfying these constraints, whether or not they arise as the optimizer for a particular data set, lead to roots whose moduli approach one as the degree increases.

For example, let
$$
a_j=\frac{1}{n},
\qquad
j=0,\ldots,n-1,
$$
so that
$$
p_n(z)
=
z^n-\frac{1}{n}\sum_{j=0}^{n-1}z^j
=
z^n-\frac{1-z^n}{n(1-z)}
$$
for $z\neq 1$. The point $z=1$ is itself a root. For any other root, multiplying by $n(1-z)$ and dividing by $z^n$ gives
$$
(n+1)-nz-z^{-n}=0.
$$

Fix $\varepsilon>0$. If $|z|\leq 1-\varepsilon$, then
$$
|z|^{-n}\geq (1-\varepsilon)^{-n},
$$
whereas
$$
|(n+1)-nz|
\leq
(n+1)+n|z|
\leq
2n+1.
$$
Because $(1-\varepsilon)^{-n}$ grows exponentially in $n$, while $2n+1$ grows only linearly, this equality cannot hold for sufficiently large $n$. Hence, for every $\varepsilon>0$, all roots of $p_n$ satisfy
$$
1-\varepsilon<|z|\leq 1
$$ for sufficiently large $n$.

Thus, for this particular distribution of nonnegative coefficients, the roots concentrate near the unit circle as $n\to\infty$. Since $p_n(1)=0$, the associated companion matrix has spectral radius one. Moreover, nonreal roots occur in complex-conjugate pairs, and roots with modulus close to one produce slowly decaying oscillatory factors in the modal representation. This provides one spectral mechanism by which recurrent qualitative behavior may persist over extended prediction horizons.
\end{remark} 

We now return to the sparse allocation used in the preceding example. The corresponding characteristic polynomial is
$$
p(z)
=
z^{256}
-
\frac{1}{3}z^{192}
-
\frac{1}{2}z^{128}
-
\frac{1}{6}z^5.
$$
Factoring out the smallest power of $z$ gives
$$
p(z)
=
z^5
\left(
z^{251}
-
\frac{1}{3}z^{187}
-
\frac{1}{2}z^{123}
-
\frac{1}{6}
\right).
$$
Thus, $z=0$ is a root of multiplicity five. Let $z\neq 0$ be any other root. Then
$$
z^{251}
-
\frac{1}{3}z^{187}
-
\frac{1}{2}z^{123}
=
\frac{1}{6},
$$
or equivalently,
$$
z^{123}
\left(
z^{128}
-
\frac{1}{3}z^{64}
-
\frac{1}{2}
\right)
=
\frac{1}{6}.
$$
Taking absolute values gives
$$
\frac{1}{6}
=
|z|^{123}
\left|
z^{128}
-
\frac{1}{3}z^{64}
-
\frac{1}{2}
\right|.
$$
By Proposition \ref{prop:positiveallocation}, $|z|\leq 1$. Therefore,
$$
\begin{aligned}
\frac{1}{6}
&\leq
|z|^{123}
\left(
|z|^{128}
+
\frac{1}{3}|z|^{64}
+
\frac{1}{2}
\right)\\
&\leq
|z|^{123}
\left(
1+\frac{1}{3}+\frac{1}{2}
\right)
=
\frac{11}{6}|z|^{123}.
\end{aligned}
$$
It follows that
$$
|z|^{123}\geq\frac{1}{11},
$$
and hence
$$
|z|
\geq
\left(\frac{1}{11}\right)^{1/123}
\approx 0.981.
$$
Consequently, although this sparse allocation produces five eigenvalues at zero, all of its nonzero eigenvalues lie in the annulus
$$
0.981\lesssim |z|\leq 1.
$$
In addition, because the coefficients sum to one, $p(1)=0$, so $z=1$ is an eigenvalue. This example illustrates how a sparse positive allocation may simultaneously produce strongly damped components and nonzero spectral components concentrated near the unit circle.

Having established the spectral consequences of the allocation constraints, we now describe the kernelized formulation used to compute the allocation coefficients in the numerical experiments. The positive-allocation problem may be posed either in the original observable coordinates or in a reproducing kernel Hilbert space. Let
$$
k:\mathbb{R}^m\times\mathbb{R}^m\to\mathbb{R}
$$ be a positive-definite kernel with associated reproducing kernel Hilbert space $\mathcal{H}_k$. There exists a feature map
$$
\varphi:\mathbb{R}^m\to\mathcal{H}_k
$$ such that
\begin{equation}
k(y_i,y_j)
=
\left\langle
\varphi(y_i),\varphi(y_j)
\right\rangle_{\mathcal{H}_k}.
\label{eq:k}
\end{equation}
The kernelized positive-allocation problem is
\begin{equation}
\min_{\alpha\in\mathbb{R}^n}
\left\|
\varphi(y_n)
-
\sum_{j=0}^{n-1}\alpha_j\varphi(y_j)
\right\|_{\mathcal{H}_k}^2
\label{eq:kernel_pa}
\end{equation}
subject to
$$
\alpha_j\geq 0,
\qquad
j=0,\ldots,n-1,
$$
and
$$
\sum_{j=0}^{n-1}\alpha_j=1.
$$

Define the Gram matrix $K\in\mathbb{R}^{n\times n}$ and the target similarity vector $k_n\in\mathbb{R}^n$ by
$$
K_{ij}=k(y_i,y_j),
\qquad
(k_n)_i=k(y_i,y_n),
\qquad
i,j=0,\ldots,n-1.
$$
Expanding the squared norm in \eqref{eq:kernel_pa} gives
$$
\begin{aligned}
\left\|
\varphi(y_n)
-
\sum_{j=0}^{n-1}\alpha_j\varphi(y_j)
\right\|_{\mathcal{H}_k}^2
&=
k(y_n,y_n)
-
2\alpha^T k_n
+
\alpha^T K\alpha.
\end{aligned}
$$
Since $k(y_n,y_n)$ is independent of $\alpha$, the kernelized allocation is equivalently obtained from the quadratic program
\begin{equation}
\min_{\alpha\in\mathbb{R}^n}
\left(
\alpha^T K\alpha-2\alpha^T k_n
\right)
\label{eq:kernel_qp}
\end{equation} subject to the same nonnegativity and normalization constraints. Thus, the allocation coefficients can be computed using kernel evaluations alone, without explicitly constructing the feature map $\varphi$.

\begin{algorithm}[!htbp]
\caption{Positive-Allocation Companion Predictor}
\label{alg:PA}
\begin{algorithmic}[1]

\Require Observable snapshots
$Y=[y_0,y_1,\ldots,y_n]\in\mathbb{R}^{m\times(n+1)}$;
prediction horizon $N_{\mathrm{pred}}$; and, for the kernelized formulation,
a positive-definite kernel $k$.

\Ensure Allocation coefficients $\alpha^*$; companion matrix $C$;
predictor eigenvalues $\{\lambda_\ell\}_{\ell=1}^n$; and model snapshots
$\{\widehat y_k\}$.

\Statex
\State \textbf{Step 1: Separate the training snapshots and target snapshot.}
\State $Y_{\mathrm{train}}\gets [y_0,y_1,\ldots,y_{n-1}]$
\State $y_{\mathrm{target}}\gets y_n$

\Statex
\State \textbf{Step 2: Compute the positive allocation.}
\If{the allocation is computed in the original observable coordinates}
    \State Find $\alpha^*\in\mathbb{R}^n$ solving
    \[
    \min_{\alpha\in\mathbb{R}^n}
    \left\|y_{\mathrm{target}}-Y_{\mathrm{train}}\alpha\right\|_2^2
    \]
    subject to
    \[
    \alpha_j\geq 0,\qquad
    \sum_{j=0}^{n-1}\alpha_j=1.
    \]
\Else
    \State Construct
    \[
    K_{ij}=k(y_i,y_j),\qquad
    (k_n)_i=k(y_i,y_{\mathrm{target}}).
    \]
    \If{regularization is used}
        \State Set $K\gets K+\eta I$, where $\eta>0$.
    \EndIf
    \State Find $\alpha^*\in\mathbb{R}^n$ solving
    \[
    \min_{\alpha\in\mathbb{R}^n}
    \left(\alpha^T K\alpha-2\alpha^T k_n\right)
    \]
    subject to
    \[
    \alpha_j\geq 0,\qquad
    \sum_{j=0}^{n-1}\alpha_j=1.
    \]
\EndIf

\Statex
\State \textbf{Step 3: Construct the companion predictor.}
\State Construct $C=C(\alpha^*)\in\mathbb{R}^{n\times n}$ according to
\eqref{eq:matrix_rep}.
\State Compute the eigenvalues
$\{\lambda_\ell\}_{\ell=1}^n=\operatorname{spec}(C)$.

\Statex
\State \textbf{Step 4: Generate the model trajectory.}
\State $e_1\gets [1,0,\ldots,0]^T\in\mathbb{R}^n$
\For{$k=0,\ldots,n+N_{\mathrm{pred}}$}
    \State $\widehat y_k\gets Y_{\mathrm{train}}C^k e_1$
\EndFor

\end{algorithmic}
\end{algorithm}

In the numerical experiments, we use the Gaussian kernel
\begin{equation}
k(y_i,y_j)
=
\exp\left(
-\frac{\|y_i-y_j\|_2^2}{2\sigma^2}
\right),
\label{eq:gaussian_kernel}
\end{equation}
where $\sigma>0$ is the kernel bandwidth. To improve numerical conditioning, the Gram matrix may be regularized according to
$$
K_{\eta}=K+\eta I,
\qquad
\eta>0,
$$ so that the quadratic term in \eqref{eq:kernel_qp} becomes $\alpha^T K_{\eta}\alpha$. The kernelized quadratic program then becomes
$$
\min_{\alpha\in\mathbb{R}^n}
\left(
\alpha^T K_\eta\alpha-2\alpha^T k_n
\right)
$$
subject to the same nonnegativity and normalization constraints. Since
$$
\alpha^T K_\eta\alpha
=
\alpha^T K\alpha+\eta\|\alpha\|_2^2,
$$ this regularization adds the small penalty $\eta\|\alpha\|_2^2$ to the allocation objective. The preceding formulations determine the allocation coefficients, while the companion matrix and model trajectory are then constructed from those coefficients. The complete procedure is summarized in Algorithm \ref{alg:PA}.

Positive allocation provides a priori information about the spectral structure of the resulting predictor. The optimized coefficient vector $\alpha^*$ determines the companion matrix and its characteristic polynomial, and the roots of that polynomial form the spectrum of the resulting finite-dimensional predictor. Consequently, the nonnegativity and normalization constraints shape the spectral structure at the construction stage, as described in Proposition \ref{prop:positiveallocation}.

When the companion matrix is diagonalizable, the same model trajectory may also be expressed through its eigenvalues and observable-space modes. The following subsection introduces this modal representation. 


\subsection{Modal Representation of the Companion Predictor}
\label{sect:modal_rep}

The companion predictor may be studied from two complementary perspectives. When $C$ is diagonalizable, its powers admit a modal representation that separates the observable-space modes from their temporal factors. Alternatively, the trajectory may be analyzed directly through powers of $C$, without diagonalizing the matrix. In exact arithmetic, the two representations describe the same model trajectory when the diagonalization exists; however, they lead to different forms of analysis and may behave differently in numerical computation. These perspectives motivate the modal and non-modal bounds developed in Section~\ref{sec:bounds}.

The following identity is a standard consequence of diagonalization and underlies modal representations used in Dynamic Mode Decomposition and Koopman spectral analysis; see \cite{TuRoLuBrKu2014, KuBrBrPr2016, HaMe2017, colbrook2024rigorous}. We include it here to establish the notation used in the subsequent analysis.

\medskip

\noindent\textbf{Proposition A.}  (Modal representation of the companion predictor). {\it 
Let
$$
Y_{\mathrm{train}}
=
[y_0,y_1,\ldots,y_{n-1}]
\in\mathbb{R}^{m\times n},
$$ and define the model trajectory by
$$
\widehat y_k
=
Y_{\mathrm{train}}C^ke_1,
\qquad k\geq 0,
$$ where $C\in\mathbb{R}^{n\times n}$ is the companion matrix associated with the allocation coefficients.

Suppose that $C$ is diagonalizable, with
\begin{equation}
C=\Gamma\Lambda\Gamma^{-1},
\label{eq:C_diag}
\end{equation}
where
$$
\Lambda
=
\operatorname{diag}(\lambda_1,\ldots,\lambda_n)
$$ and
$$
\Gamma
=
[v_1\ \cdots\ v_n].
$$ Define
$$
b:=\Gamma^{-1}e_1
$$ and
$$
\psi_\ell:=Y_{\mathrm{train}}v_\ell
\in\mathbb{C}^m.
$$ Then the model trajectory admits the exact modal representation
\begin{equation}
\widehat y_k
=
\sum_{\ell=1}^n
b_\ell\lambda_\ell^k\psi_\ell.
\label{eq:modal_expansion}
\end{equation}
}

\begin{proof}
Using the diagonalization of $C$, we obtain
$$
C^ke_1
=
\Gamma\Lambda^k\Gamma^{-1}e_1
=
\sum_{\ell=1}^n
b_\ell\lambda_\ell^kv_\ell.
$$ Multiplying by $Y_{\mathrm{train}}$ gives
$$
\widehat y_k
=
Y_{\mathrm{train}}C^ke_1
=
\sum_{\ell=1}^n
b_\ell\lambda_\ell^k
Y_{\mathrm{train}}v_\ell
=
\sum_{\ell=1}^n
b_\ell\lambda_\ell^k\psi_\ell.
$$
\end{proof}

\begin{remark}
The temporal factors $\lambda_\ell^k$ form a Vandermonde-type structure. For a collection of prediction indices $k=0,\ldots,N$, the modal trajectory may therefore be represented using a matrix of eigenvalue powers. This form separates the spatial or observable-space modes $\psi_\ell$ from their temporal evolution.
\end{remark}

\section{Modal and Non-modal Bounds}
\label{sec:bounds}

The companion predictor constructed in Section~\ref{sect:pa} via \eqref{eq:Cpred} generates a discrete model trajectory
$$
\widehat y_k=Y_{\mathrm{train}}C^ke_1.
$$
When $C$ is diagonalizable, the modal representation from Section~\ref{sect:modal_rep} allows finite differences of this model trajectory to be computed mode by mode. These finite differences provide a useful lens for interpreting derivative-like quantities computed from the model trajectory: they may emphasize oscillatory or changing components of the prediction while suppressing modes near $\lambda_\ell=1$. 

For a time step $\Delta t>0$, define the forward first and second finite differences by
$$
D_+\widehat y_k
:=
\frac{\widehat y_{k+1}-\widehat y_k}{\Delta t},
$$
and
$$
D_+^2\widehat y_k
:=
\frac{\widehat y_{k+2}-2\widehat y_{k+1}+\widehat y_k}{\Delta t^2}.
$$

\begin{proposition}[Modal representation of forward differences]
\label{prop:modal_forward_differences}
Suppose that the companion matrix $C$ is diagonalizable and that the model trajectory admits the modal representation
$$
\widehat y_k
=
\sum_{\ell=1}^n b_\ell \lambda_\ell^k \psi_\ell,
\qquad
k\geq 0,
$$ where $\lambda_\ell$ are the eigenvalues of $C$, $b_\ell$ are the modal coefficients, and $\psi_\ell$ are the associated observable-space modes. Then the forward first difference satisfies
\begin{equation}
D_+\widehat y_k
=
\frac{1}{\Delta t}
\sum_{\ell=1}^n
b_\ell
(\lambda_\ell-1)
\lambda_\ell^k
\psi_\ell,
\label{eq:first_difference_modal}
\end{equation}
and the forward second difference satisfies
\begin{equation}
D_+^2\widehat y_k
=
\frac{1}{\Delta t^2}
\sum_{\ell=1}^n
b_\ell
(\lambda_\ell-1)^2
\lambda_\ell^k
\psi_\ell.
\label{eq:second_difference_modal}
\end{equation}
Consequently,
\begin{equation}
\|D_+\widehat y_k\|_2
\leq
\frac{1}{\Delta t}
\sum_{\ell=1}^n
|b_\ell|\,|\lambda_\ell-1|\,|\lambda_\ell|^k
\|\psi_\ell\|_2,
\label{eq:first_difference_bound}
\end{equation}
and
\begin{equation}
\|D_+^2\widehat y_k\|_2
\leq
\frac{1}{\Delta t^2}
\sum_{\ell=1}^n
|b_\ell|
|\lambda_\ell-1|^2
|\lambda_\ell|^k
\|\psi_\ell\|_2.
\label{eq:second_difference_bound}
\end{equation}
\end{proposition}

\begin{proof}
Using the modal representation, we compute
$$
\widehat y_{k+1}-\widehat y_k
=
\sum_{\ell=1}^n
b_\ell
\left(
\lambda_\ell^{k+1}-\lambda_\ell^k
\right)
\psi_\ell
=
\sum_{\ell=1}^n
b_\ell
(\lambda_\ell-1)
\lambda_\ell^k
\psi_\ell.
$$
Dividing by $\Delta t$ gives \eqref{eq:first_difference_modal}.

Similarly,
$$
\widehat y_{k+2}-2\widehat y_{k+1}+\widehat y_k
=
\sum_{\ell=1}^n
b_\ell
\left(
\lambda_\ell^{k+2}
-
2\lambda_\ell^{k+1}
+
\lambda_\ell^k
\right)
\psi_\ell.
$$
Factoring each modal term gives
$$
\lambda_\ell^{k+2}
-
2\lambda_\ell^{k+1}
+
\lambda_\ell^k
=
(\lambda_\ell-1)^2\lambda_\ell^k.
$$
Dividing by $\Delta t^2$ gives \eqref{eq:second_difference_modal}. The norm bounds \eqref{eq:first_difference_bound} and \eqref{eq:second_difference_bound} follow by taking Euclidean norms and applying the triangle inequality.
\end{proof}

Equations~\eqref{eq:first_difference_modal} and \eqref{eq:second_difference_modal} show that finite differences of the model trajectory act as spectral filters. Each mode is weighted by $\lambda_\ell-1$ in the first difference and by $(\lambda_\ell-1)^2$ in the second difference. Thus, modes with $\lambda_\ell$ close to $1$ contribute weakly to the forward differences, while modes with nonzero phase contribute oscillatory finite-difference behavior. This provides a useful lens for interpreting derivative-like diagnostics computed from the model trajectory, especially when pointwise positional agreement degrades. In Section~\ref{sect:numerical}, we use this perspective to interpret the observed behavior of the first- and second-difference diagnostics.

We now record an internal growth bound for the constructed model trajectory. This bound describes the size of $\widehat y_k$ in terms of the modal representation of the companion predictor.

\begin{proposition}[Modal growth bound]
\label{prop:modal_growth}
Suppose that the companion matrix $C$ is diagonalizable and that the model trajectory admits the modal representation
$$
\widehat y_k
=
\sum_{\ell=1}^n b_\ell \lambda_\ell^k \psi_\ell,
\qquad
k\geq 0.
$$
Let
$$
\rho=\max_{1\leq \ell\leq n}|\lambda_\ell|.
$$
Define the observable-mode matrix
$$
M=
[\psi_1 \ \psi_2 \ \cdots \ \psi_n]
$$
and the modal coefficient vector
$$
b=[b_1,\ldots,b_n]^T.
$$
Then, with respect to the Euclidean norm,
\begin{equation}
\|\widehat y_k\|_2
\leq
\|M\|_2 \|b\|_2 \rho^k.
\label{eq:modal_growth}
\end{equation}
\end{proposition}

\begin{proof}
The modal representation may be written in matrix form as
$$
\widehat y_k
=
M\Lambda^k b,
$$
where
$$
\Lambda=\operatorname{diag}(\lambda_1,\ldots,\lambda_n).
$$
Therefore,
$$
\|\widehat y_k\|_2
\leq
\|M\|_2
\|\Lambda^k\|_2
\|b\|_2.
$$
Since $\Lambda$ is diagonal,
$$
\|\Lambda^k\|_2
=
\max_{1\leq \ell\leq n}|\lambda_\ell|^k
=
\rho^k.
$$
This gives
$$
\|\widehat y_k\|_2
\leq
\|M\|_2\|b\|_2\rho^k,
$$
as claimed.
\end{proof}

\begin{remark}
For model trajectories obtained from normalized positive allocation, Proposition~\ref{prop:positiveallocation} implies that all eigenvalues of the companion matrix lie in the closed unit disk. Moreover, since the allocation coefficients sum to one, the associated characteristic polynomial satisfies $p(1)=0$. Hence $1$ is an eigenvalue of $C$, and the spectral radius of the companion predictor is exactly one:
$$
\rho(C)=1.
$$
Thus, the modal growth bound does not imply exponential growth for positive-allocation predictors. The size of the model trajectory is then controlled by the observable-mode matrix $M$ and the modal coefficient vector $b$, together with possible nonnormal or conditioning effects.
\end{remark}

The observable-space modes in this bound are determined by the training snapshots. Indeed, from Section~\ref{sect:modal_rep},
$$
\psi_\ell=Y_{\mathrm{train}}v_\ell,
$$
where $v_\ell$ is an eigenvector of the companion matrix $C$. Therefore each mode $\psi_\ell$ lies in the column span of $Y_{\mathrm{train}}$.
If
$$
\Gamma=[v_1 \ \cdots \ v_n],
$$
then
$$
M=[\psi_1 \ \cdots \ \psi_n]
=
Y_{\mathrm{train}}\Gamma.
$$
Consequently,
$$
\|M\|_2
\leq
\|Y_{\mathrm{train}}\|_2\|\Gamma\|_2.
$$
This expresses the modal growth bound in terms of the observed training data and the conditioning of the eigenvector matrix of the companion predictor.

The preceding modal identities describe internal properties of the constructed model trajectory. They explain how the companion spectrum enters the finite differences and growth behavior of the predictor. However, these internal bounds do not by themselves imply accuracy with respect to the original nonlinear dynamics. We therefore record a standard discrete error decomposition that separates amplification under the true dynamics from the one-step residual of the model trajectory.

\begin{proposition}[Discrete prediction error decomposition]
\label{prop:error_decomposition}
Let $F:\Omega\to\Omega$ be the true discrete-time state map, and let
$$
x_{k+1}=F(x_k)
$$
be the true trajectory. Suppose that $F$ is Lipschitz with constant $L>0$, so that
$$
\|F(x)-F(z)\|_2 \le L\|x-z\|_2
$$
for all relevant states $x,z\in\Omega$.

Assume that the physical state coordinates are included among the
observables, or that a fixed reconstruction map from observable
snapshots to state snapshots has been specified. Let $\widehat x_k$ denote the predicted state recovered from the model
trajectory $\widehat y_k$. Define the prediction error
$$
e_k := \|x_k-\widehat x_k\|_2
$$
and the one-step residual
$$
\delta_k := \|F(\widehat x_k)-\widehat x_{k+1}\|_2.
$$
Then, for $k\geq 1$,
\begin{equation}
e_k
\leq
L^{k-1}e_1
+
\sum_{m=1}^{k-1}
L^{k-1-m}\delta_m.
\label{eq:errorbound}
\end{equation}
\end{proposition}

\begin{proof}
For each $m\geq 1$, we have
$$
\begin{aligned}
e_{m+1}
&=
\|x_{m+1}-\widehat x_{m+1}\|_2  \\
&=
\|F(x_m)-\widehat x_{m+1}\|_2 \\
&\leq
\|F(x_m)-F(\widehat x_m)\|_2
+
\|F(\widehat x_m)-\widehat x_{m+1}\|_2 \\
&\leq
L e_m+\delta_m.
\end{aligned}
$$
Iterating this inequality gives
$$
e_k
\leq
L^{k-1}e_1
+
\sum_{m=1}^{k-1}
L^{k-1-m}\delta_m,
$$
which proves the claim.
\end{proof}
The bound \eqref{eq:errorbound} separates two sources of prediction error. The term $L^{k-1}e_1$ describes the amplification of the initial prediction error by the true nonlinear dynamics, while the residual terms $\delta_m$ measure the one-step discrepancy between the true dynamics and the model-generated trajectory. Thus, even when the companion predictor has controlled internal spectral behavior, errors may accumulate if the underlying dynamics are sensitive or if the model trajectory is not approximately invariant under $F$.

The modal results above provide interpretable bounds in terms of eigenvalues, modal coefficients, and observable-space modes. Their use, however, relies on diagonalizability of the companion matrix and may be sensitive to the conditioning of the eigenvector matrix. We now turn to a complementary $C$-based analysis that avoids modal decomposition and bounds finite differences directly through powers of the companion matrix.


\subsection{$C$-Based Non-Modal Finite-Difference Bounds}
\label{sect:nonmodal}

We now derive finite-difference bounds that avoid modal decomposition. Unlike the modal bounds above, these estimates do not require the companion matrix $C$ to be diagonalizable. Instead, they follow directly from the model trajectory
$$
\widehat y_k=Y_{\mathrm{train}}C^ke_1
$$ and the submultiplicative property of matrix norms.

\begin{proposition}[$C$-based finite-difference bounds]
\label{prop:C_based_bounds}
Let
$$
\widehat y_k=Y_{\mathrm{train}}C^ke_1,
\qquad k\geq 0,
$$
where $C\in\mathbb{R}^{n\times n}$ is the companion matrix, $Y_{\mathrm{train}}\in\mathbb{R}^{m\times n}$ is the matrix of training snapshots, and $e_1=(1,0,\ldots,0)^T\in\mathbb{R}^n$. For $\Delta t>0$, define the forward finite differences
$$
D_+\widehat y_k
:=
\frac{\widehat y_{k+1}-\widehat y_k}{\Delta t},
$$
and
$$
D_+^2\widehat y_k
:=
\frac{\widehat y_{k+2}-2\widehat y_{k+1}+\widehat y_k}{\Delta t^2}.
$$
Then
\begin{equation}
\|D_+\widehat y_k\|_2
\le
\frac{1}{\Delta t}
\|Y_{\mathrm{train}}\|_2
\|C^k(C-I)e_1\|_2,
\label{eq:Cbound_first}
\end{equation}
and
\begin{equation}
\|D_+^2\widehat y_k\|_2
\leq
\frac{1}{\Delta t^2}
\|Y_{\mathrm{train}}\|_2
\left\|
C^k(C-I)^2e_1
\right\|_2.
\label{eq:Cbound_second}
\end{equation}
\end{proposition}

\begin{proof}
Using the model trajectory, we have
$$
\widehat y_{k+1}-\widehat y_k
=
Y_{\mathrm{train}}
\left(
C^{k+1}-C^k
\right)e_1.
$$
Factoring gives
$$
\widehat y_{k+1}-\widehat y_k
=
Y_{\mathrm{train}}C^k(C-I)e_1.
$$
Taking Euclidean norms and applying submultiplicativity of the induced
matrix norm gives
$$
\|\widehat y_{k+1}-\widehat y_k\|_2
\le
\|Y_{\mathrm{train}}\|_2
\left\|
C^k(C-I)e_1
\right\|_2.
$$
Dividing by $\Delta t$ proves \eqref{eq:Cbound_first}.

Similarly,
$$
\widehat y_{k+2}-2\widehat y_{k+1}+\widehat y_k
=
Y_{\mathrm{train}}
\left(
C^{k+2}-2C^{k+1}+C^k
\right)e_1.
$$
Since
$$
C^{k+2}-2C^{k+1}+C^k
=
C^k(C-I)^2,
$$
we obtain
$$
\widehat y_{k+2}-2\widehat y_{k+1}+\widehat y_k
=
Y_{\mathrm{train}}C^k(C-I)^2e_1.
$$
Taking norms and dividing by $\Delta t^2$ gives
\eqref{eq:Cbound_second}.
\end{proof}

The bounds \eqref{eq:Cbound_first} and \eqref{eq:Cbound_second} can be evaluated directly from $Y_{\mathrm{train}}$, $C$, and $e_1$, without computing eigenvalues or eigenvectors. They therefore remain valid even when $C$ is defective or when its eigenvector matrix is poorly conditioned. In the numerical experiments, these quantities provide a non-modal diagnostic for the size of the first and second finite differences of the constructed model trajectory.

For visualization, we compare the observed finite-difference norms
$$
\|D_+\widehat y_k\|_2
\quad\text{and}\quad
\|D_+^2\widehat y_k\|_2
$$
with the corresponding $C$-based upper bounds
$$
\frac{1}{\Delta t}
\|Y_{\mathrm{train}}\|_2
\|C^k(C-I)e_1\|_2
$$
and
$$
\frac{1}{\Delta t^2}
\|Y_{\mathrm{train}}\|_2
\|C^k(C-I)^2e_1\|_2.
$$
These quantities are plotted on a logarithmic scale to compare their relative size across prediction time steps.


\section{Numerical Results and Model Diagnostics}
\label{sect:numerical}

This section examines the positive-allocation companion predictor on two nonlinear dynamical systems with different qualitative behavior: the FitzHugh--Nagumo system and the susceptible--infectious--recovered (SIR) epidemiological model. The FitzHugh--Nagumo system provides an oscillatory test case with recurrent structure, while the SIR model provides a transient, dissipative test case with monotone approach toward equilibrium.

For each system, we construct the companion predictor from observed trajectory data and evaluate the resulting model trajectory, spectral structure, modal activity, and finite-difference diagnostics. The goal of these experiments is to illustrate how the positive-allocation construction behaves in systems with contrasting dynamics and to examine how the modal and $C$-based diagnostics developed in Section~\ref{sec:bounds} appear in practice. Together, the examples highlight both the interpretability and the limitations of the method, including the effects of recurrence, dissipation, scaling, and numerical conditioning.

\subsection{Positive-Allocation Companion Prediction for the FitzHugh--Nagumo System}
\label{subsec:fhn}
\begin{figure}[t]
    \centering
    \includegraphics[width=0.95\textwidth]{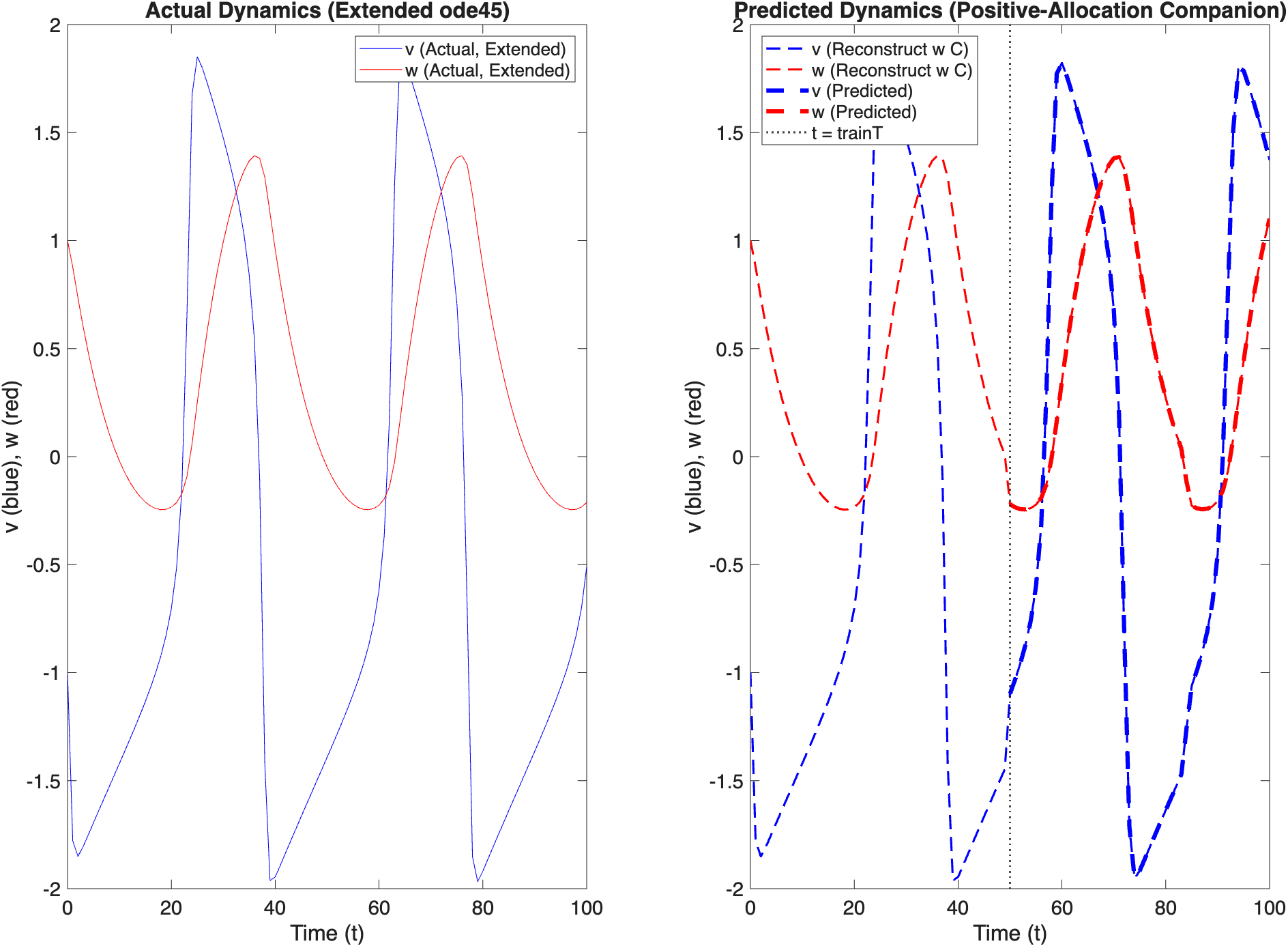} 
    \caption{Comparison of simulated and positive-allocation companion predicted dynamics for the FitzHugh--Nagumo system. The model trajectory preserves several qualitative features of the oscillatory dynamics, including repeated spiking behavior and the coupled evolution of the voltage-like variable $v(t)$ and recovery variable $w(t)$. A phase shift develops over the extended prediction horizon. The vertical dotted line marks the transition from training to prediction at $t=50$.}
	\label{fig:FHNDynamics}
\end{figure}

The FitzHugh--Nagumo system models the evolution of a voltage-like excitation variable $v(t)$ and a recovery variable $w(t)$ in excitable media, often used as a simplification of the Hodgkin-Huxley model for neuronal spiking dynamics \cite{fitzhugh1961impulses,nagumo1962active}. The system is given by
\begin{align}
\frac{dv}{dt} &= v - \frac{v^3}{3} - w + I, \\
\frac{dw}{dt} &= \varepsilon (v + a - b w),
\end{align} where $I$ is a constant stimulus current, and $\varepsilon, a, b$ are positive parameters governing time-scale separation and nullcline shapes. 

The simulated training data were collected with selected parameter values $\varepsilon = 0.08, a=0.7, b=0.8$ and $I=0.5$ which yield a regime in which the system exhibits relaxation oscillations as shown on the left-hand plot of Figure \ref{fig:FHNDynamics}. The state trajectories exhibit periodic spiking behavior, with $v(t)$ acting as the fast variable and $w(t)$ as the slow recovery component. The solution was sampled at intervals of $\Delta t=1$ over a training interval sufficient to capture multiple oscillatory cycles. 

The results in Figure \ref{fig:FHNDynamics} were obtained by collecting 51 snapshots over the training interval $[0,50].$ The final training snapshot was approximated as a convex combination of the preceding $50$ snapshots, and the resulting allocation coefficients were used to construct a companion matrix $C\in\mathbb R^{50\times 50}$. The model trajectory was then generated by repeated application of the companion predictor, as in \eqref{eq:Cpred}. See Section~\ref{sect:pa} for the positive-allocation formulation and the kernelized implementation used in the numerical experiments.

Before applying the positive-allocation procedure, the observable data were centered and scaled by observable coordinate. After prediction, the model trajectory was transformed back to the original coordinate scale for comparison with the simulated FitzHugh--Nagumo trajectory. In this example, scaling improved the qualitative behavior of the prediction even though it also increased the condition number of the companion matrix $C$ (e.g., $\approx 3.3\times 10^{12}$ ompared with $1.8\times 10^4$ without scaling). Without scaling, it is possible that higher-order nonlinear terms can dominate numerically even if they are not dynamically dominant. This illustrates that conditioning alone does not fully determine predictive behavior; the interaction among scaling, allocation weights, recurrence structure, and spectral distribution also affects the resulting model trajectory. 

The recurrent structure of the FitzHugh--Nagumo trajectory makes it a natural test case for positive allocation. Since the method approximates the final training snapshot using a convex combination of earlier snapshots, recurrent dynamics increase the likelihood that earlier states contain useful information for constructing the recurrence. Thus, the example provides a useful setting for examining how near-unit spectral components and sparse snapshot contributions affect long-horizon prediction.

Figure~\ref{fig:FHNDynamics} compares the model trajectory generated by the positive-allocation companion predictor with the simulated FitzHugh--Nagumo trajectory. The prediction reproduces several qualitative features of the oscillatory dynamics, including the repeated spiking structure and the coupled behavior of $v(t)$ and $w(t)$. A phase shift is visible as the prediction evolves, indicating that pointwise trajectory agreement degrades over longer horizons. Nevertheless, the model preserves important recurrent features beyond the training window, which motivates the finite-difference and spectral diagnostics examined below. The visible transition at the end of the training interval at $t=50$ marks the point at which the companion recurrence begins generating the model trajectory beyond the observed data. 

\begin{figure}[ht!]
    \centering     \includegraphics[width=\textwidth]{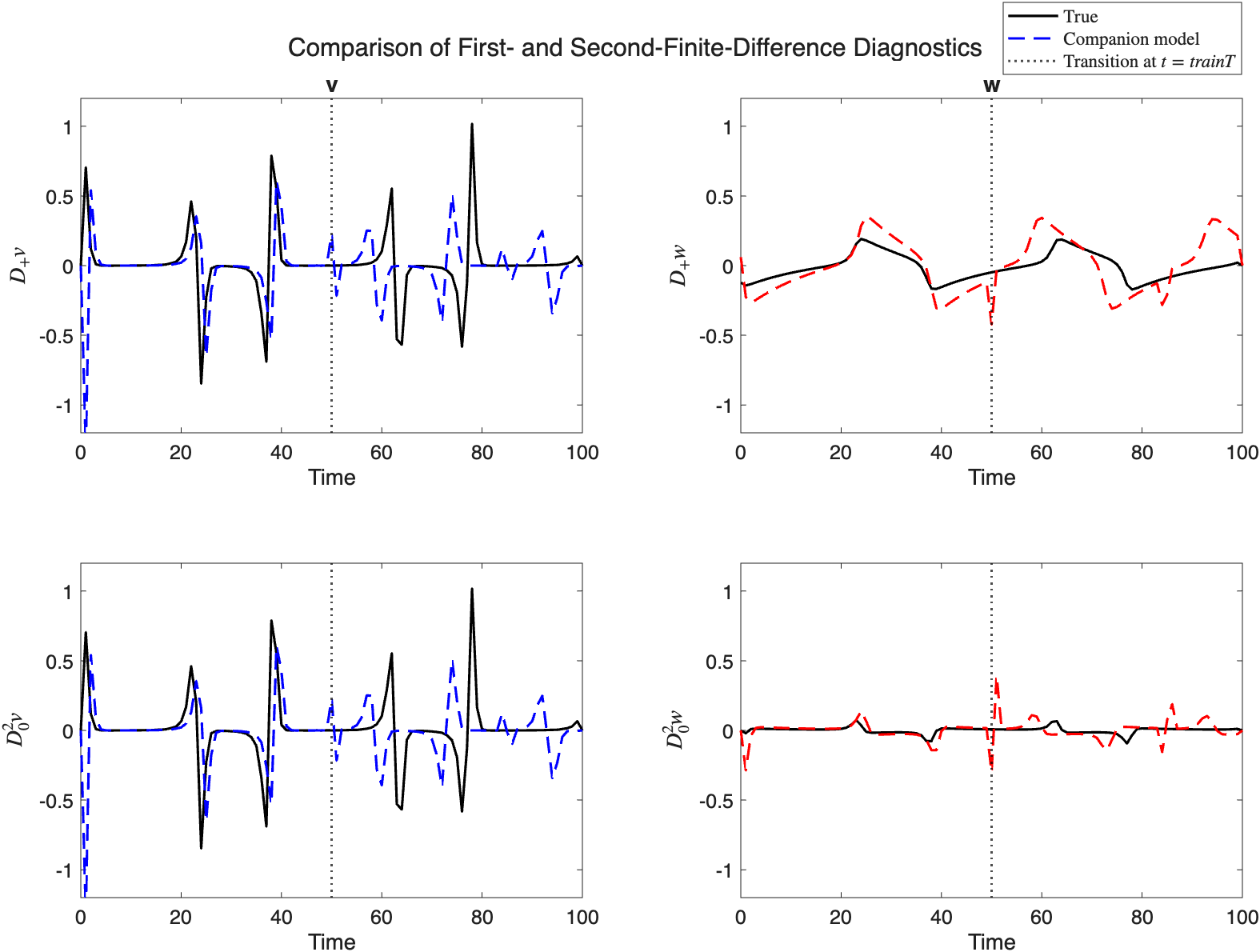} 
    \caption{Comparison of first- and second-order finite-difference diagnostics for the FitzHugh--Nagumo system. The diagnostics are computed from the simulated trajectory and from the positive-allocation companion model trajectory. The vertical dotted line marks the transition from training to prediction at $t=50$.}
    \label{fig:FHNderivcompare}
\end{figure}


\subsubsection{Finite-Difference Diagnostics and $C$-Based Estimates}

The positive-allocation companion predictor was further evaluated using first- and second-order finite-difference diagnostics. These quantities were computed from the model trajectory and compared with corresponding finite-difference quantities computed from the simulated FitzHugh--Nagumo trajectory. As shown in Figure~\ref{fig:FHNderivcompare}, the first finite difference of the model trajectory captures much of the oscillatory structure of the system, especially in the voltage-like variable $v(t)$, although phase drift accumulates over longer prediction horizons. The second finite difference is more sensitive to this drift and exhibits larger discrepancies, particularly in the recovery variable $w(t)$.

\begin{figure}[h]
    \centering
    \includegraphics[width=\textwidth]{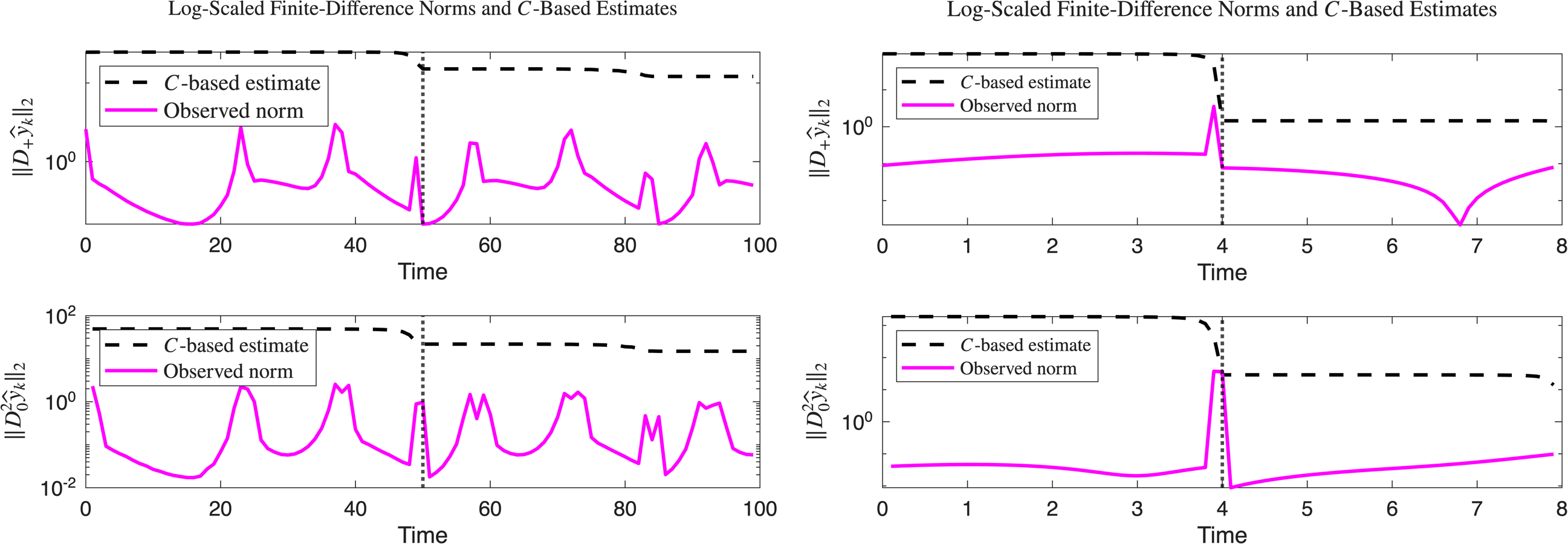}
    \caption{Log-scaled comparison of finite-difference norms and corresponding
    $C$-based diagnostic quantities for (a) FitzHugh--Nagumo on the left and (b) SIR model on the right. The vertical markers indicate the transition from training to prediction. These plots provide non-modal diagnostics of the companion-generated recurrence rather than direct measures of pointwise prediction accuracy.}
    \label{fig:logscale-bounds}
\end{figure}

To interpret these diagnostics, we also evaluate the non-modal $C$-based bounds from Proposition~\ref{prop:C_based_bounds}. These bounds are computed directly from $Y_{\mathrm{train}}$, $C$, and $e_1$, and therefore do not require diagonalization of the companion matrix. Figure~\ref{fig:logscale-bounds}a compares the observed finite-difference norms with their corresponding $C$-based upper bounds on a logarithmic scale. The model finite-difference norms remain below the bound envelope in this experiment, while the gap between the observed quantities and the bounds reflects the conservatism expected from norm-based estimates. This comparison provides a diagnostic for the size of the first and second finite differences generated by the companion predictor, rather than a direct guarantee of pointwise trajectory accuracy.

\begin{remark}
The positive-allocation companion construction is sensitive to implementation choices such as temporal resolution, scaling, and the length of the training window. In the FitzHugh--Nagumo experiments, decreasing the time step or extending the training window can increase the condition number of the companion matrix and degrade the numerical behavior of the resulting prediction. In exact arithmetic, the eigenvalues of the companion matrix coincide with the roots of its characteristic polynomial; therefore, any observed discrepancy between computed eigenvalues and computed characteristic roots should be interpreted as a numerical conditioning effect rather than a mathematical inconsistency. These observations suggest that the training window and sampling resolution should be treated as part of the modeling procedure rather than as purely incidental numerical choices.
\end{remark}

\begin{remark}
The analytical finite-difference identities in Section~\ref{sec:bounds} are stated using forward differences. In the numerical implementation, the first finite-difference diagnostic is computed using the forward difference
$$
\frac{\widehat y_{k+1}-\widehat y_k}{\Delta t},
$$
while the second finite-difference diagnostic is computed using the
centered difference
$$
\frac{\widehat y_{k+1}-2\widehat y_k+\widehat y_{k-1}}{\Delta t^2}.
$$
This centered second difference corresponds to the same $C$-based
second-difference expression up to an index shift, since
$$
\widehat y_{k+1}-2\widehat y_k+\widehat y_{k-1}
=
Y_{\mathrm{train}}C^{k-1}(C-I)^2e_1.
$$
Thus, the numerical diagnostic is consistent with the non-modal $C$-based analysis, with the time index shifted relative to the forward second-difference formula. Because finite-difference stencils involve neighboring snapshots, diagnostic quantities that use the first recurrence-generated snapshot may show a one-index transition relative to the vertical marker used for the state trajectory.
\end{remark}


\subsection{Positive-Allocation Companion Prediction for the SIR Model}
\label{subsec:sir}
\begin{figure}[ht]
    \centering
    \includegraphics[width=0.95\textwidth]{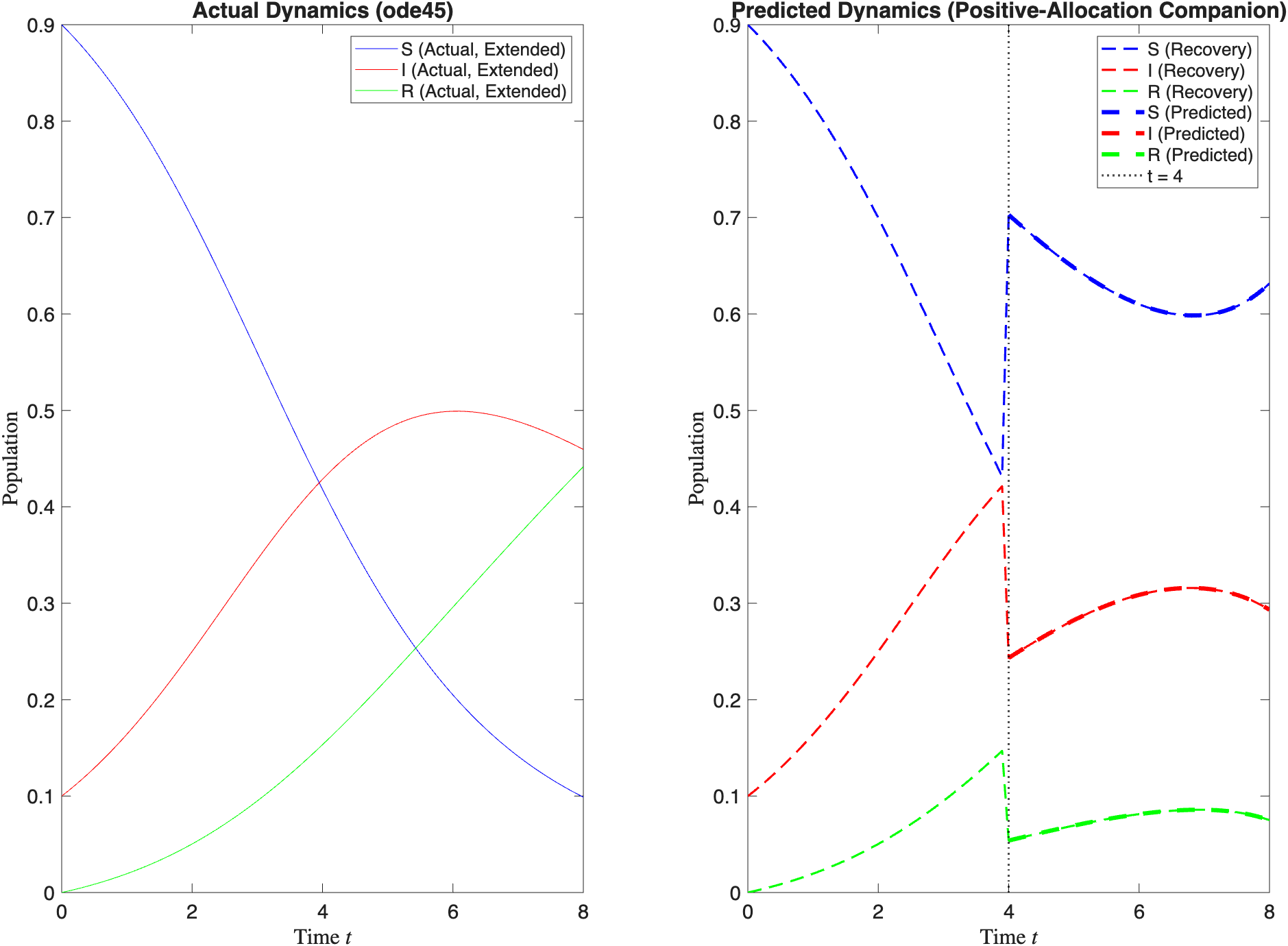} 
    \caption{Comparison of simulated and positive-allocation companion predicted dynamics for the SIR epidemiological model. The model reflects some qualitative features including the transient rise and subsequent decline of the infectious component, but pointwise agreement degrades over the prediction horizon. The vertical dotted line marks the transition from training to prediction at $t=4$.}
	\label{fig:SIRDynamics}
\end{figure}
We now apply the positive-allocation companion predictor to the classical susceptible--infectious--recovered (SIR) epidemiological model, originating in the work of Kermack and McKendrick~\cite{kermack1927contribution}. This example provides a contrast with the oscillatory FitzHugh--Nagumo system: rather than exhibiting recurrent dynamics, the SIR trajectory contains transient growth in the infectious population followed by dissipative approach toward equilibrium. The governing equations are given by
\begin{equation}
\begin{aligned}
\frac{dS}{dt} &= -\beta S I, \\
\frac{dI}{dt} &= \beta S I - \gamma I, \\
\frac{dR}{dt} &= \gamma I,
\end{aligned}
\end{equation}
where $S(t)$, $I(t)$, and $R(t)$ denote the susceptible, infectious, and recovered populations at time $t$, and with parameter values $\beta=0.75$ and $\gamma=0.15$ which represent transmission and recovery rates, respectively. The solution was sampled at intervals of $\Delta t=0.1$ over the extended interval $[0,8]$, while the training interval is taken to be $[0,4]$.

To construct the companion predictor, we embed the simulated trajectory into the observable space
$$
\mathbf g(S,I,R) =
[S,\ I,\ R,\ I^2]^T.
$$ This observable set augments the state variables with a simple nonlinear term that provides sensitivity to infectious growth while keeping the observable dimension small. The final training snapshot is approximated by positive allocation using the preceding training snapshots, and the resulting coefficients are used to construct the companion matrix $C$. The model trajectory is then generated by repeated application of the companion predictor, as in \eqref{eq:Cpred}.

Figure~\ref{fig:SIRDynamics} compares the model trajectory with the simulated SIR dynamics. The prediction reflects some qualitative structure of the susceptible, infectious, and recovered populations, including the transient growth and subsequent decline of the infectious component. As in the FitzHugh--Nagumo example, the transition from training to prediction is marked by the vertical line. The SIR example is useful because it tests the positive-allocation construction in a setting where the dynamics are not recurrent in the same way as an oscillatory system.

We also compare first- and second-order finite-difference diagnostics computed from the model trajectory with corresponding finite-difference quantities computed from the simulated SIR trajectory. In the implementation, the first finite difference is computed using a forward difference, while the second finite difference is computed using a centered difference, as described in the implementation remark above.

These diagnostics describe the rate-of-change and curvature behavior of the generated model trajectory. In the SIR example, this distinction is important because pointwise agreement of the predicted state variables is limited over the full prediction horizon. The finite-difference diagnostics therefore provide a complementary view of the recurrence: they indicate whether the companion predictor captures aspects of the local temporal structure, such as growth, decay, and curvature, even when the positional trajectory itself is not accurately tracked. In this sense, the SIR example illustrates both a limitation of the predictor and the usefulness of derivative-like diagnostics for understanding how the constructed model trajectory evolves.

Figure~\ref{fig:logscale-bounds}b compares the observed finite-difference norms with the corresponding $C$-based diagnostic quantities on a logarithmic scale. The first-difference quantity is computed using the forward matrix difference
$$
C^k-C^{k-1},
$$
while the second-difference quantity is computed using
$$
C^{k+1}-2C^k+C^{k-1},
$$
matching the centered second-difference diagnostic used in the implementation. These quantities are evaluated directly from the companion matrix and the training data, without diagonalizing $C$. The comparison provides a non-modal diagnostic for the finite-difference behavior induced by the recurrence.


\subsection{Comparative Diagnostics Across Systems}
The FitzHugh--Nagumo and SIR examples illustrate different behaviors of the positive-allocation companion predictor. The FitzHugh--Nagumo trajectory is recurrent and oscillatory, while the SIR trajectory is transient and dissipative. These differences are reflected in both the temporal modal amplitudes and the spectra of the associated companion matrices.

\begin{figure}[!htbp]
    \centering
    \includegraphics[width=\textwidth]{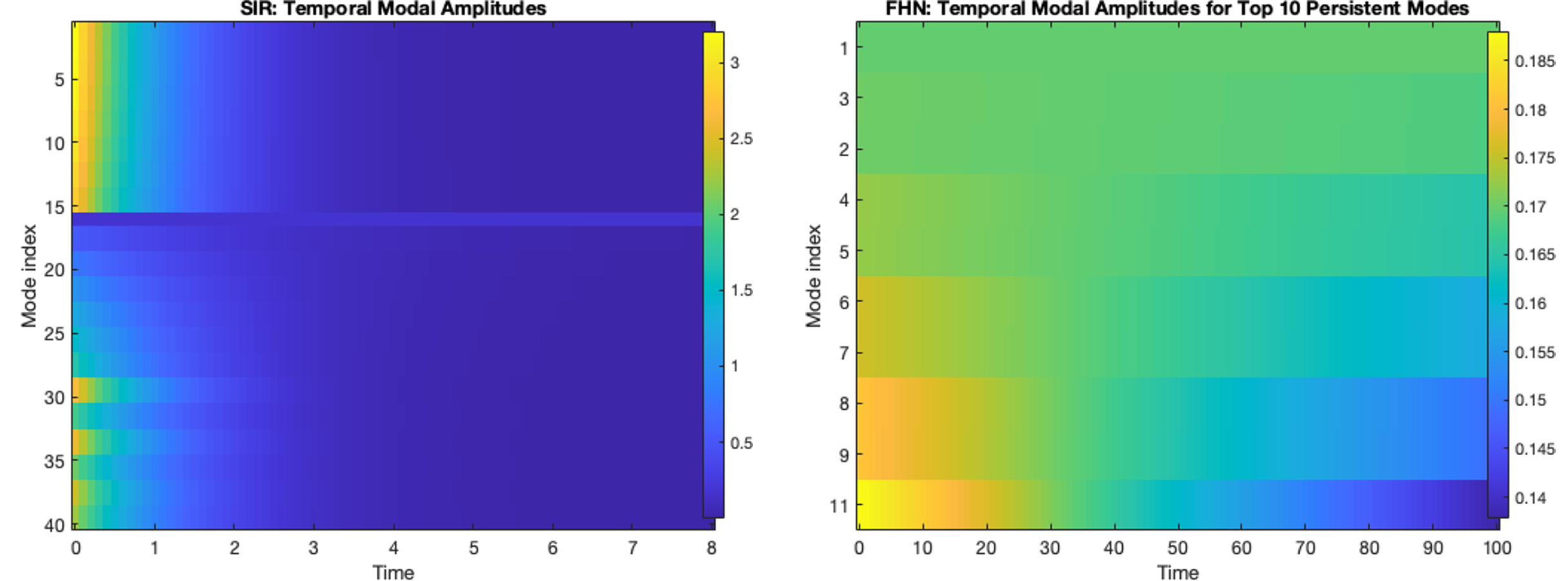}
    \caption{Temporal modal amplitudes for the SIR and FitzHugh--Nagumo companion predictors. (a) The SIR example on the left exhibits broader and more transient modal activity. (b) The FitzHugh--Nagumo example on the right shows persistent activity in a smaller collection of modes. These plots diagnose the companion-generated recurrence and should not be interpreted as direct measures of pointwise prediction accuracy.}
    \label{fig:modal-comparison}
\end{figure}

Figure~\ref{fig:modal-comparison} compares the temporal modal amplitudes for the two systems. For the FitzHugh--Nagumo example, a small collection of modes remains active over the prediction window, consistent with the persistence of oscillatory structure in the model trajectory. In contrast, the SIR example exhibits broader and more transient modal activity. Several modes contribute over different time segments, reflecting the nonrecurrent growth-and-decay structure of the epidemic trajectory.

This comparison should be interpreted as a diagnostic of the constructed recurrence rather than as a direct measure of prediction accuracy. In particular, the SIR example shows that modal or finite-difference structure may remain informative even when pointwise trajectory agreement is limited. The temporal modal amplitudes help identify which spectral components are active in the companion-generated trajectory, while the finite-difference diagnostics describe how the recurrence organizes rate-of-change and curvature-like behavior.

\begin{figure}[!htbp]
    \centering
    \includegraphics[width=\textwidth]{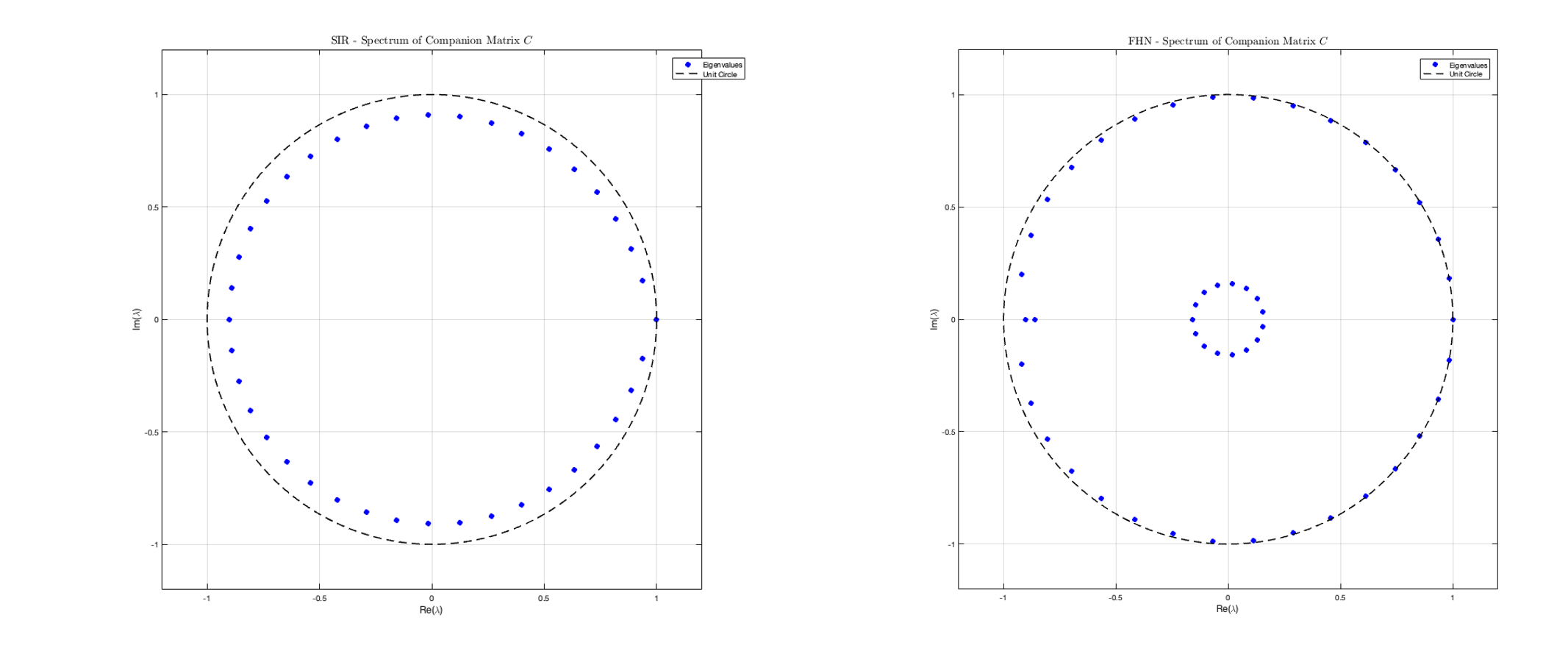}
    \caption{Spectra of the positive-allocation companion matrices for (a)  the SIR on the left and (b) the
    FitzHugh--Nagumo, on the right, examples.
    In both cases, normalized positive allocation places the spectrum in the closed unit disk and includes the eigenvalue $1$. The FitzHugh--Nagumo spectrum contains more near-unit oscillatory components, while the SIR spectrum contains more visibly decaying components, reflecting the different qualitative behavior of the two systems.}
    \label{fig:SpectCompare}
\end{figure}

The spectral plots in Figure~\ref{fig:SpectCompare} provide a complementary view. For normalized positive allocation, Proposition \ref{prop:positiveallocation} implies that all companion eigenvalues lie in the closed unit disk, and the coefficient normalization implies that $1$ is an eigenvalue. Thus, the distinction between the two examples is not whether the spectra lie inside or outside the unit disk, but how the non-unit spectral components are distributed. The FitzHugh--Nagumo predictor has many eigenvalues near the unit circle, which is consistent with slowly decaying oscillatory temporal factors. The SIR predictor has more visibly decaying components, consistent with transient dynamics and approach toward equilibrium.

Taken together, these diagnostics suggest that the positive-allocation companion predictor is most naturally suited to systems in which the final training snapshot can be meaningfully represented by earlier trajectory information. Recurrent systems such as FitzHugh--Nagumo are therefore favorable examples for this construction. Transient systems such as SIR are more challenging: the positional prediction may degrade, but the spectral, modal, and finite-difference diagnostics still provide useful information about the behavior of the constructed recurrence.


\section{Conclusion}
\label{sec:conclusions}

This work introduced a positive-allocation companion construction for Koopman-inspired finite-dimensional prediction of nonlinear dynamical systems. Rather than fitting a finite-dimensional linear operator and then examining its spectrum as a post-hoc diagnostic, the proposed construction determines recurrence coefficients through nonnegative, normalized allocation weights. These coefficients define a companion matrix whose spectral structure is shaped by the allocation constraints at the construction stage.

The main analytical consequence of this formulation is that normalized positive allocation produces a companion spectrum contained in the closed unit disk. Since the allocation coefficients sum to one, the associated characteristic polynomial also has $1$ as a root, so the companion predictor has spectral radius one in exact arithmetic. This spectral containment does not by itself guarantee accurate prediction of the underlying nonlinear dynamics, but it provides useful structural information about the internally generated model trajectory.

We also developed modal and non-modal diagnostics for the companion predictor. When the companion matrix is diagonalizable, the model trajectory admits a modal representation in which finite differences act as spectral filters through factors of $\lambda_\ell-1$ and $(\lambda_\ell-1)^2$. To avoid reliance on diagonalization, we also derived $C$-based finite-difference bounds that can be evaluated directly from the training data, the companion matrix, and the initial coordinate vector. These bounds and diagnostics describe properties of the constructed recurrence and should be distinguished from guarantees of pointwise accuracy with respect to the original nonlinear system.

The numerical examples illustrate both the strengths and limitations of the approach. For the FitzHugh--Nagumo system, the recurrent oscillatory structure is well matched to the positive-allocation construction, and the resulting companion predictor preserves important qualitative features of the oscillatory dynamics over the prediction window. For the SIR model, the transient and dissipative dynamics present a more challenging setting: pointwise trajectory agreement is limited over the full horizon, but the finite-difference and modal diagnostics still provide information about how the constructed recurrence organizes growth, decay, and curvature-like behavior.

These examples indicate that the method is most naturally suited to settings in which the final training snapshot can be meaningfully represented using earlier trajectory information. The quality of the resulting predictor depends on the observable selection, temporal resolution, scaling, training-window length, and numerical conditioning of the companion construction. Future work will investigate adaptive choices of observables and training windows, more systematic comparisons with DMD, EDMD, and kernel-based Koopman methods, and sharper error estimates relating the internal companion dynamics to the original nonlinear system.

\backmatter

\bmhead{Acknowledgements}

The authors would like to thank the anonymous reviewers of an earlier version of this manuscript for their helpful comments. Part of this research was performed while the authors were visiting the Institute for Pure and Applied Mathematics (IPAM), which is supported by the National Science Foundation (Grant No. DMS-1925919), for the Long Program on the Mathematics of Intelligence. The first author received support through the National Science Foundation PRIMES program (No. 2424984).

\section*{Declarations}
\textbf{Code availability. } The numerical examples in this paper use simulated data generated from the FitzHugh--Nagumo and SIR systems described in Section~\ref{sect:numerical}. Code used to generate the numerical results is available from the corresponding author upon reasonable request.

\bibliography{sn-bibliography}

\end{document}